\documentclass[a4paper,12pt]{article}
\usepackage{amsmath}
\usepackage{amssymb}
\usepackage{setspace}
\usepackage{fullpage}
\usepackage{bbm}
\usepackage{amsthm}
\newtheorem{theorem}{Theorem}[section]
\newtheorem{lemma}[theorem]{Lemma}

\newtheorem{corollary}[theorem]{Corollary}
\newtheorem{result}[theorem]{Result}
\theoremstyle{definition}

\newtheorem*{definition}{Definition}

\newtheorem*{remark}{Remark}

\usepackage{pdfsync}
\def\PG{\mathrm{PG}} \def\AG{\mathrm{AG}} \def\AGL{\mathrm{AGL}}
  \def\Persp{\mathrm{Persp}}

\def\Aut{\mathrm{Aut}}
\def\PGammaL{\mathrm{P}\Gamma\mathrm{L}}
\def\PGL{\mathrm{PGL}}

\def\A{\mathcal{A}}  \def\C{\mathcal{C}}
  
 \def\K{\mathcal{K}}
\def\L{\mathcal{L}}  \def\N{\mathcal{N}}
\def\O{\mathcal{O}} \def\P{\mathcal{P}}

\def\PGO{\mathrm{PGO}} 

\def\PGammaO{\mathrm{P}\Gamma\mathrm{O}}
\def\AGammaL{\mathrm{A}\Gamma\mathrm{L}}
\def\Sz{\mathrm{Sz}}

\def\F{\mathbb{F}}

\title{A construction for infinite families of semisymmetric
graphs revealing their full automorphism group}
\author{Philippe Cara\thanks{Partially supported by grant 15.263.08 of the `Fonds Wetenschappelijk Onderzoek-Vlaanderen'.} \and Sara Rottey \and Geertrui Van de Voorde\thanks{Supported by VUB-grant GOA62.}}
\begin{document}
\maketitle
\begin{abstract}
We give a general construction leading to different non-isomorphic families $\Gamma_{n,q}(\K)$ of connected $q$-regular semisymmetric
graphs of order $2q^{n+1}$ embedded in $\PG(n+1,q)$, for a prime power $q=p^h$, using the {\em linear representation} of a particular point set $\K$ of size $q$ contained in a hyperplane of $\PG(n+1,q)$. We show that, when $\K$ is a {\em normal rational curve} with one point removed, the graphs $\Gamma_{n,q}(\K)$ are isomorphic to the graphs constructed for $q$ prime in \cite{Du} and to the graphs constructed for $q=p^h$ in \cite{Felix}.
These graphs were known to be semisymmetric but their full
automorphism group was up to now unknown. For $q\geq n+3$ or $q=p=n+2$, $n\geq 2$, we
obtain their full automorphism group from our construction by showing that, for an {\em arc} $\K$, every automorphism
of $\Gamma_{n,q}(\K)$ is induced by a collineation of the ambient space $\PG(n+1,q)$.
We also give some other examples of semisymmetric graphs $\Gamma_{n,q}(\K)$ for which not every automorphism is induced by a collineation of their ambient space.
\end{abstract}

\section{Introduction}
In the following, all graphs are assumed to be simple, i.e. they are
undirected graphs which contain no loops or multiple edges.

\begin{definition}
We say a graph is {\em vertex-transitive} if its automorphism group
acts transitively on the vertices. Similarly, a graph is
{\em edge-transitive} if its automorphism group acts transitively on
the edges. A graph is {\em semisymmetric} if it is regular and
edge-transitive but not vertex-transitive (see~\cite{Folkman}).
\end{definition}

One can easily prove that a semisymmetric graph must be bipartite
with equal partition sizes. Moreover the automorphism group must be
transitive on both partition sets. General constructions of
semisymmetric graphs are quite rare.

We construct several infinite families $\Gamma_{n,q}(\K)$ of
semisymmetric graphs using affine points and some selected lines of a
projective space $\PG(n+1,q)$. The infinite series of semisymmetric
graphs given in~\cite{Du} is shown to be part of the
same series as in~\cite{Felix} which is exactly one of the families that we will construct in this paper (see Section~\ref{isomorphisms}). Using
our construction, in many cases, the structure of the full automorphism group of the
graphs $\Gamma_{n,q}(\K)$ can be clarified (at least for $q\geq n+3$ and $q=p=n+2$). This structure was not given
in~\cite{Du,Felix} where only part of the automorphism group is
constructed, enough to show edge-transitivity.

\section{Construction and properties of the graph $\Gamma_{n,q}(\K)$}\label{section2}

\begin{definition} Let $\K$ be a point set in $H_\infty=\PG(n,q)$ and embed $H_\infty$ in $\PG(n+1,q)$. The {\em linear representation} $T_n^*(\K)$ of $\K$ is a point-line incidence structure with natural incidence, point set $\P$ and line set $\L$ as follows:
\begin{itemize}
\item[$\P$:] {\em affine} points of $\PG(n+1,q)$  (i.e. the points of $\PG(n+1,q)\setminus H_\infty$),
\item[$\L$:] lines of $\PG(n+1,q)$ through a point of $\K$, but not lying in $H_\infty$.
\end{itemize}
\end{definition}
For more information on linear representations of geometries, we refer to \cite{declerck}.

\begin{definition} We denote the point-line incidence graph of $T_n^*(\K)$ by $\Gamma_{n,q}(\K)$, i.e. the bipartite graph with classes $\P$ and $\L$ and adjacency corresponding to the natural incidence of the geometry.
\end{definition}
Whenever we consider the incidence graph $\Gamma_{n,q}(\K)$ of some linear representation $T_n^*(\K)$ of $\K$, we still regard the set of vertices as a set of points and lines in $\PG(n+1,q)$. In this way we can use the inherited properties of this space and borrow expressions such as the span of points, a subspace, incidence, and others.

We define the closure of a set of points in $\PG(n,q)$ as follows:

\begin{definition}
Consider a set $S$ of points in $\PG(n,q)$. We construct the {\em
closure} $\overline{S}$ of $S$ recursively as follows:
\begin{itemize}
\item[(i)] determine the set $\A$ of all subspaces of $\PG(n,q)$ spanned by an arbitrary number of points of $S$;
\item[(ii)] determine the set $\overline{S}$ of points $P$ for which there exist two subspaces in $\A$ that intersect only at $P$; if $\overline{S} \neq S$ replace $S$ by
$\overline{S}$ and go to~(i), otherwise stop.
\end{itemize}
\end{definition}

This definition corresponds to the definition of a closure of a set of points in a plane in
\cite[Chapter XI]{Hughes}. Here the authors show that if $S$ is contained in a plane and contains a quadrangle, the points of $\overline{S}$ form the
smallest subplane of $\PG(2,q)$ containing all points of $S$.

Similarly, if $S$ contains a frame of $\PG(n,q)$, the points of $\overline{S}$ form the
smallest $n$-dimensional subgeometry of $\PG(n,q)$ containing all points of $S$.

\begin{result}\label{connected}{\rm \cite[Corollary 4.3]{bart}} The graph $\Gamma_{n,q}(\K)$ is connected if and only if the span $\langle \K\rangle$ has dimension $n$.
\end{result}

\begin{remark}
Suppose the set $\K$ spans a $t$-dimensional subspace $\PG(t,q)$ of $H_\infty=\PG(n,q)$, $t<n$. One can check that in this case the graph $\Gamma_{n,q}(\K)$ is a non-connected graph with $q^{n-t}$ connected components, where each component is isomorphic to the graph $\Gamma_{t,q}(\K)$. This clarifies why we will only consider graphs $\Gamma_{n,q}(\K)$ with set $\K$ such that $\langle \K \rangle = H_\infty$.
\end{remark}

Throughout this paper, we use the following theorems of \cite{wij2}.

\begin{result}
Let $|\K|\neq q$ or let $\K$ be a set of $q$ points of $H_\infty$ such that every point of $H_\infty \backslash \K$ lies on at least one tangent line to $\K$. Suppose $\alpha$ is an isomorphism between $\Gamma_{n,q}(\K)$ and $\Gamma_{n,q}(\K')$, for some set $\K'$ in $H_\infty$, then $\alpha$ stabilises $\P$.
\end{result}
\begin{corollary}\label{NVT}
If $\K$ is a set of $q$ points of $H_\infty$ such that every point of $H_\infty$ lies on at least one tangent line to $\K$, then $\Gamma_{n,q}(\K)$ is not vertex-transitive.
\end{corollary}

%\begin{remark} If $\K$ does not satisfy the condition in the previous corollary, then their exist counterexamples to the corollary. Let $\K$ be the $q$-arc $\{(0,1,x,x^2)\mid x \in \F_q\}$, $q$ even, embedded in the plane $H$ of $\PG(3,q)$ with equation $X_0=0$. Consider the mapping $\phi$ from the affine point $(1,a,b,c)$ to the line $\langle (0,1,a,a^2),(1,0,c,b^2)\rangle$. It is not hard to see that, if $q$ is even, $\phi$ maps collinear points to concurrent lines and vice versa, hence, preserves the edges of the graph $\Gamma_{3,q}(\K)$ but switches the sets $\P$ and $\L$. It will follow from the arguments of this paper that this graph is edge-transitive. Together with the existence of the mapping $\phi$, this yields that $\Gamma_{3,q}(\K)$ is vertex-transitive.
%\end{remark}
%%%

\begin{result}\label{isomgraphs}Let $q>2$. Let $\K$ and $\K'$ be sets of $q$ points such that $\overline{\K}$ is equal to $H_\infty$ and such that every point of $H_\infty$ lies on at least one tangent line to $\K$. Consider an isomorphism $\alpha$ between $\Gamma_{n,q}(\K)$ and $\Gamma_{n,q}(\K')$. Then $\alpha$ is induced by an element of the stabiliser $\PGammaL(n+2,q)_{H_\infty}$ mapping $\K$ onto $\K'$.
\end{result}

\begin{result} \label{main}Let $q>2$ and let $\K$ be a set of $q$ points such that $\overline{\K}$ is equal to $H_\infty$ and such that every point of $H_\infty$ lies on at least one tangent line to $\K$. Then $\Aut(\Gamma_{n,q}(\K)) \cong \PGammaL(n+2,q)_\K$.
\end{result}

%\begin{result} Let $\K$ denote a point set in $H_\infty=\PG(n-1,q)$ such that the closure $\overline{\K}$ is equal to $\PG(n-1,q)$, and such that every point of $H_\infty$ lies on at least one tangent line to $\K$. Then $\Aut(\Gamma_{n,q}(\K))=\PGammaL(n+1,q)_\K$.
%\end{result}

Recall that if a group $G$ has a normal subgroup $N$ and the quotient $G/N$ is isomorphic
to some group $H$, we say that $G$ is an {\em extension} of $N$ by $H$. This is written as $G = N.H$.

An extension $G = N.H$ which is a semidirect product is also called a {\em split
extension}. This means that one can find a subgroup $\overline{H}\cong H$ in $G$ such that $G = N\overline{H}$ and
$N \cap \overline{H} = \{e_G \}$ and is denoted by $G=N \rtimes H$.

If the set of elements of $\PGammaL(n+2,q)$ fixing all points of
the hyperplane $H_\infty$ is written as $\Persp(H_\infty)$, then
$\Persp(H_\infty)$ consists of all elations and homologies with axis
$H_\infty$. Clearly it has size $|\Persp(H_\infty)|= q^{n+1}(q-1)$.

\begin{result}\label{isom} If the setwise stabilisers $\PGammaL(n+1,q)_\K$ and $\PGL(n+1,q)_\K$, respectively, of a point set $\K$ spanning $H_\infty=\PG(n,q)$ fixes a point of $H_\infty$, then
$\PGammaL(n+2,q)_\K \cong \Persp(H_\infty)\rtimes \PGammaL(n+1,q)_\K$ and $\PGL(n+2,q)_\K \cong \Persp(H_\infty)\rtimes \PGL(n+1,q)_\K$, respectively.
\end{result}

\begin{result}\label{isom2}
If the setwise stabiliser $\PGL(n+1,q)_\K$ of a point set $\K$ spanning $H_\infty=\PG(n,q)$, $q=p^h$, fixes a point of $H_\infty$, and $\PGammaL(n+1,q)_\K \cong \PGL(n+1,q)_\K \rtimes \Aut(\F_{q_0})$, for $q_0=p^{h_0}$, $h_0 | h$ or $\PGammaL(n+1,q)_\K \cong \PGL(n+1,q)_\K $, then $\PGammaL(n+2,q)_\K \cong \Persp(H_\infty)\rtimes \PGammaL(n+1,q)_\K$.
\end{result}

The following theorem is easy to prove. We will use it to show the edge-transitivity of the constructed graphs.

\begin{theorem} \label{ET} If the stabiliser $\PGammaL(n+1,q)_\K$ of $\K$ in the full collineation group of $H_\infty$ acts transitively on the points of $\K$, then $\Gamma_{n,q}(\K)$ is an edge-transitive graph.
\end{theorem}
\begin{proof}
Consider two edges $(R_i,L_i)$,
$i=1,2$, where $R_i\in \P$, $L_i \in \L$, $R_i\in L_i$. Let $P_i$ be
$L_i\cap H_\infty$. Since $\PGammaL(n+1,q)_\K$ is transitive on $\K$, we
may take an element $\beta$ of $\PGammaL(n+1,q)_\K$ such that $\beta(P_1)=P_2$. This element extends to an element $\beta'$ of $(\PGammaL(n+2,q)_{H_\infty})_\K$ mapping $P_1$ onto $P_2$. %Hier moeten we niet voorzichtig zijn: die extensie is niet per se uniek maar dat kan geen kwaad. We hebben dus niet nodig dat de extensie splitst.

Let $S$ be the point at infinity of the line $\beta'(R_1)R_2$, then
there is a (unique) elation $\gamma$ with centre $S$ and axis
$H_\infty$ mapping $\beta'(R_1)$ to $R_2$. This elation maps
$\beta'(L_1)$ onto $L_2$. Since $\gamma\circ\beta'$ is an element of $(\PGammaL(n+2,q)_{H_\infty})_\K$
mapping $(R_1,L_1)$ onto $(R_2,L_2)$, the statement follows.
\end{proof}

The main goal of this paper is the construction of infinite families of semisymmetric graphs. The results of \cite{wij2} introduced in this section will enable us to explicitly describe the automorphism group of the constructed graphs. Note that, since a semisymmetric graph is regular, any graph $\Gamma_{n,q}(\K)$ that is semisymmetric, necessarily has $|\K|=q$. For this reason, we will investigate point sets of size $q$ in $\PG(n,q)$. Moreover, considering Result \ref{main} and Theorem \ref{ET}, we will look for point sets $\K$ such that the closure $\overline{\K}$ is equal to $H_\infty$ and such that $\PGammaL(n+1,q)_\K$ acts transitively on the points of $\K$. 

We give a brief overview of all constructions to come.
\begin{table}[h!]
\begin{tabular}{|l|l|l|l|}\hline
$\K$ & Condition & $|\Aut(\Gamma_{n,q}(\K))|$ &  Reference \\
\hline
basis   &   $q=n+1$                 & $hq^{n+1}(q-1)q!$       & \textsection  3.1 \\ %$\mathrm{Sym}(q)\rtimes \Aut(\F_q)$
frame &   $q=n+2$                   &  $hq^{n+1}(q-1)^{n}q!$  & \textsection 3.1, \cite{Felix}\\ %$\mathrm{PMon}(q)\rtimes\Aut(\F_q)$
 $\subset$ NRC  &   $q\geq n+3$          & $hq^{n+2}(q-1)^{2}$     & \textsection 3.2, \cite{Felix}, \cite{Du} ($q=p$)\\
 $\subset$ non-classical arc & $q>4$ even& $hq^{5}(q-1)^{2}$       & \textsection 3.3 \\
$\subset$ Glynn-arc &  $q=9$      & $9^6 8^2$              & \textsection 3.4 \\
$\subset Q^{-}(3,q)$ & $q>4$ square      & $> hq^5(q-1)^2$          & \textsection 4.1\\
 $\subset$ Tits-ovoid  &   $q=2^{2(2e+1)}$   & $> hq^5(q-1)(\sqrt{q}-1) $& \textsection 4.2 \\
$\subset Q^{+}(3,q)$   &  $q>4$ square   & $> 2hq^5(q-1)(\sqrt{q}-1)^2$ & \textsection 4.3\\
  $\subset$ cone $V\O$ & $q={q_0}^h$&  $> hq^{2n+1}(q-1)^2 |\PGammaL(n,q_0)_{\O}|$ & \textsection 4.4 \\
\hline
\end{tabular}
%\caption{}\label{Table}
\end{table}

When the size of the automorphism group is given, all automorphisms are geometric, i.e. induced by a collineation of the ambient space. If the size is larger than a given bound, this means there exist automorphisms that are not geometric.

\section{Families of semisymmetric graphs arising from arcs}\label{constructions}
We are in search of point sets $\K$ such that the closure $\overline{\K}$ is equal to $H_\infty$ and such that $\PGammaL(n+1,q)_\K$ acts transitively on the points of $\K$. An arc of size $q$ turns out to be an excellent choice for the point set $\K$.

\begin{definition}
A {\em $k$-arc} in $\PG(n,q)$ is a set of $k$ points, $k \geq n+1$,
such that no $n+1$ points lie on a hyperplane. \end{definition}

If $\A$ is a $k$-arc in $\PG(n,q)$, then $k\geq n+1$, hence, we will only consider the case where $q\geq n+1$. If $q=n+1$, then it is easy to see that an arc of size $q$ in $\PG(n,q)$ is a basis, if $q=n+2$, then every arc of size $q$ is a frame. Hence, there are no non-isomorphic arcs of size $q$ in $\PG(n,q)$ when $q=n+1$ or $q=n+2$. Because of the isomorphism of the graph $\Gamma_{n,q}(\K)$ with other graphs (see Section \ref{isomorphisms}), we will explicitly investigate these cases, but the more interesting examples occur when $q\geq n+3$.

It is conjectured that an arc in $\PG(n,q)$, $3\leq n\leq q-3$, has at most $q+1$ points (this is the well-known MDS-conjecture, in view of its coding-theoretical description). An example of an arc of size $q+1$ is given by the {\em normal rational curve}.

\begin{definition}
A {\em normal rational curve}, NRC for short, in $\PG(n,q)$, $2 \leq n \leq q$, is
a $(q+1)$-arc projectively equivalent to the $(q+1)$-arc
$\{(0,\ldots,0,1)\}\cup\{  (1,t,t^2,t^3,\ldots,t^{n}) \mid t\in \F_q\}$
\cite[Section 27.5]{GGG}.
\end{definition}

\begin{remark}  There are results showing that, if $n$ is sufficiently large w.r.t. $q$, an arc of size $q$ in $\PG(n,q)$ can be extended to an arc of size $q+1$. Moreover, other results show that for many values of $q$ and $n$, all $(q+1)$-arcs in $\PG(n,q)$ are normal rational curves. The combination of these results leads to the understanding why there are not many known examples of $q$-arcs in $\PG(n,q)$ that are not contained in a normal rational curve. For an overview, we refer to \cite{packing}.
\end{remark}

We will construct different families of graphs, arising from non-isomorphic arcs of size $q$. Hence, it follows from Result \ref{isomgraphs} that the obtained graphs are non-isomorphic.

In view of Result \ref{main}, our first goal is to show that the closure of a set of $q$ points of an arc in $\PG(n,q)$, $q\geq n+3$ or $q=p=n+2$ prime, is $H_\infty$. When $n=2$, this follows immediately. In the following lemmas, we deal with the case $n\geq 3$.
%%%%

%DIT BEWIJS HOORT HIER NIET ECHT THUIS....MAAR MISSCHIEN NOG NUTTIG LATER OF IN HET ANDER ARTIKEL
%\begin{lemma} \label{n3} If $n=3$, then all points of $H_\infty$ are rigid.
%\end{lemma}
%\begin{proof} The closure of a $q$-arc in $\PG(2,q)$, $q>3$ is clearly the plane $\PG(2,q)$ itself. Now consider the case $q=3$. Denote the points of $\K$ by $P_1,P_2,P_3$. From Lemmas \ref{crucial} and \ref{opspanning}, we get that the affine points of a plane $\pi$ through $P_1P_2$ are mapped onto the affine points of a plane through $\a(P_1)\a(P_2)$. Let $R,S$ and $T$ be the affine points of a line in $\pi$, not through $P_1$ or $P_2$. It is clear that the point $\a(T)$ cannot be contained in $\a(R)\a(P_1)$ or $\a(R)\a(P_2)$, hence, since $q=3$, the point $\a(T)$ must be the unique affine point on the line $\a(R)\a(S)$. If $M$ and $N$ are two lines in $\pi$ through the same point $Q$ of $P_1P_2$, then this together with the fact that $\a$ preserves distance in $\Gamma_{n,q}(\K)$ also implies that $\infty(\a(M))=\infty(\a(N))$, hence, that the point $Q$ is rigid. We have showed that all points on the lines $P_1P_2,P_1P_3$ and $P_2P_3$ are rigid. The statement now follows from the fact that the closure of this set of rigid points is the plane $H_\infty$.
%\end{proof}

\begin{lemma} \label{derdepunt} Let $\K$ be an arc of size $q$ in $\PG(n,q)$, $n \geq 3$. Let $P_1$ and $P_2$ be any two points of $\K$;
\begin{description}
\item[] if $q=n+2$, there is at least one additional point in $\overline{\K}$ (the closure of $\K$) on the
line $P_1P_2$,
\item[] if $q \geq n+3$, there are at least $q/2$ additional points in $\overline{\K}$ on the line $P_1P_2$.
\end{description}
\end{lemma}
\begin{proof}
Note that a $k$-space $\pi$, $k\leq n-2$, with $k+1$ points of
$\mathcal{K}$, different from $P_1$ and $P_2$, does not intersect $P_1
P_2$, since otherwise $\langle \pi, P_1 P_2 \rangle$ would be a
$(k+1)$-space containing $k+3$ points of $\mathcal{K}$, contradicting the arc condition.

Let $P_3,\ldots,P_{n+2}$ be $n$ points of $\K$, different from $P_1$
and $P_2$. The space $\langle P_3,\ldots,P_{n+2}\rangle$ is a hyperplane of
$H_\infty$, hence, it meets the line $P_1P_2$ in a point $Q$. This
point $Q$ is contained in $\overline{\K}$ but not contained in $\K$ since $\K$ is an arc. If $q=n+2$, there is exactly one set
$\{P_3,\ldots,P_{n+2}\}$ of $n$ points of $\mathcal{K}$, different
from $P_1$ and $P_2$, yielding an extra point in $\overline{\K}$ on $P_1P_2$.

If $n+3 \leq q \leq 2n+2$, then let $\{P_3, \ldots, P_{n+3}\}$ be a set
of $n+1$ points of $\mathcal{K}$, different from $P_1$ and $P_2$. Any
subset with $n$ points of $\{P_3, \ldots, P_{n+2}\}$ defines a
hyperplane intersecting $P_1 P_2$ in a point $Q \neq P_1, P_2$ contained in $\overline{\K}$.
These points $Q$ are all different since any two considered hyperplanes
intersect in a $(n-2)$-space with $n-1$ points of $\mathcal{K}$, and
hence this space does not intersect $P_1 P_2$. There are $n+1$ such
subsets, so the line $P_1 P_2$ contains $q/2 \leq n+1 \leq q-2$ additional points in $\overline{\K}$ different from $P_1$ and $P_2$.

If $q\geq 2n+2$, then let $P_3, \ldots, P_{n+1}$ be $n-1$ points of
$\mathcal{\K}$, different from $P_1$ and $P_2$. Clearly $\left \langle
P_3, \ldots, P_{n+1}\right \rangle$  is disjoint from $P_1 P_2$.
There are  $q-n-1$ points of $\mathcal{\K}$ different from all $P_i$,
$i=1,\ldots, n+1$. For every such point $R$, the hyperplane $\left
\langle P_3, \ldots, P_{n+1}, R\right \rangle$ intersects $P_1 P_2$ in a
point of $\overline{\K}$ different from $P_1$ and $P_2$. Again, all these points are
different since two such hyperplanes intersect in $\left \langle P_3,
\ldots, P_{n+1}\right \rangle$. The line $P_1 P_2$ contains $q-n-1 \geq q/2$ points of $\overline{\K}$ different from $P_1$ and $P_2$.
\end{proof}

\begin{lemma}\label{driepunten} Let $\K$ be an arc of size $q$ in $\PG(n,q)$. Let $q\geq n+3$ or $q=p=n+2$, $n\geq 2$ and let $\mu_\infty$ be a plane containing $3$ points of $\K$. Then every point of $\mu_\infty$ is contained in $\overline{\K}$.
\end{lemma}
\begin{proof} Let $P_1,P_2,P_3$ be $3$ points of $\K$ and let $\mu_\infty$ be
the plane $\langle P_1,P_2,P_3\rangle = \PG(2,q)$. Consider $q\geq n+3$. By Lemma
\ref{derdepunt}, we know that there exist at least $q/2$ points in $\overline{\K}$
on each of the lines $P_2P_3$, $P_1P_3$ and $P_1P_2$, different from $P_1$,
$P_2$ and $P_3$. Consider the set $S$ containing all these points and
points $P_1$, $P_2$ and $P_3$. Its closure $\overline{S}$ forms a
subplane $\pi$ of $\mu_\infty$ consisting of only points of $\overline{\K}$. Since a proper subplane of $\PG(2,q)$ contains at
most $\sqrt{q}+1 < q/2+2$ points of the line $P_1 P_2$, we see that
$\pi$ must be $\mu_\infty$. If $q=n+2$ is prime, by Lemma \ref{derdepunt}, we
find an extra point $Q_i \in \overline{\K}$, $i=2,3$, on the line $P_1P_i$. The closure of
$\{P_1, P_2, P_3, Q_2, Q_3\}$ forms a subplane with all points
in $\overline{\K}$. By the fact that $q$ is prime, this sublane equals $\mu_\infty = \PG(2,q)$.
\end{proof}

\begin{lemma} \label{vlak}Let $L$ be a line such that every point is
in $\overline{\K}$, let $\pi_\infty$ be a plane of $H_\infty$ through $L$, containing
at least two points $R_1$ and $R_2$ of $\overline{\K}$ outside $L$. Then every point in the plane $\pi_\infty$ is
in $\overline{\K}$.
\end{lemma}

\begin{proof} The closure of the set of points of $\overline{\K}$ on the line $L$,
together with the points $R_1$ and $R_2$ is clearly the plane
$\pi_\infty$ itself.
\end{proof}

\begin{lemma} \label{closure} Let $q\geq n+3$ or $q=p=n+2$, $n\geq 2$ and let $\K$ be an arc of size $q$ in $\PG(n,q)$, then $\overline{\K}=\PG(n,q)$.
\end{lemma}
\begin{proof} For $n=2$, this easily follows. Let $P_1,\ldots,P_{q}$ be the points of $\K$. By Lemma
 \ref{driepunten}, we know that every point of $\langle
 P_1,P_2,P_3\rangle$ is in $\overline{\K}$. Suppose, by induction, that every point
 in $\langle P_1,\ldots,P_k\rangle$, $k\leq n$ is in $\overline{\K}$. The point $P_{k+1}$
 is not contained in $\langle P_1,\ldots,P_k\rangle$. There exists an additional
 point $Q$ in $\overline{\K}$ on the line $P_1P_{k+1}$ by Lemma \ref{derdepunt}. Let
 $S$ be a point of $\langle P_1,\ldots,P_{k+1}\rangle$, not on the line $P_1P_{k+1}$, and let $R$ be
 the intersection of the line $SP_{k+1}$ with $\langle
 P_1,\ldots,P_k\rangle$. Since every point on the line $RP_1$ is in $\overline{\K}$, and $\langle
 RP_1,P_{k+1}\rangle$ contains the points $Q$ and $P_{k+1}$ of $\overline{\K}$,
 Lemma \ref{vlak} implies that the point $S$ is in $\overline{\K}$, as are the points of $P_1P_{k+1}$. This shows
 that every point in $\langle P_1,\ldots,P_{k+1}\rangle$ is in $\overline{\K}$. The
lemma follows by induction and the fact that
 $H_\infty=\langle P_1,\ldots,P_{n+1}\rangle$.
\end{proof}

\begin{theorem} \label{hoofdstelling}  Let $q\geq n+3$ or $q=p=n+2$, $n\geq 3$, or $n=2$ and $q$ odd, and let $\K$ be an arc in $\PG(n,q)$, then
$\Aut(\Gamma_{n,q}(\K))\cong \PGammaL(n+2,q)_\K$.
\end{theorem}
\begin{proof}It is clear that every point of $H_\infty$ lies on a tangent line to the arc. By Lemma \ref{closure}, $\overline{\K}$ equals $\PG(n,q)$. The theorem follows from Result \ref{main}.
\end{proof}

%%%%

\subsection{$\K$ is a $q$-arc in $\PG(n,q)$ with $q=n+1$ or $q=n+2$}\label{q=n}
As noted before, a $q$-arc in $\PG(n,q)$ with $q=n+1$ is a basis, a $q$-arc in $\PG(n,q)$ with $q=n+2$ is a frame. In these cases, the linear representation of a $q$-arc gives rise to a semisymmetric graph, however, the description of the automorphism group is different from the case $q\geq n+3$. In the following proof, we cannot use the same techniques as in \cite{wij2} to show that $\PGammaL(n+2,q)_\K$ splits over $\Persp(H_\infty)$.

\begin{theorem}\label{apart} If $\K$ is a $q$-arc in $\PG(n,q)$, $q=n+1$ or $q=n+2$, with $(n,q)\neq (2,4)$ then $\Gamma_{n,q}(\K)$ is a semisymmetric graph. The group $\PGammaL(n+2,q)_\K$ is a subgroup of $\Aut(\Gamma_{n,q}(\K))$ and is isomorphic to $\Persp(H_\infty)\rtimes \PGammaL(n+1,q)_\K$, where $\PGammaL(n+1,q)_\K$ is isomorphic to
\begin{itemize}
\item[(i)] $\mathrm{Sym}(q)\rtimes \Aut(\F_q)$ if $q=n+2$, having size $hq^{n+1}(q-1)q!$;
\item[(ii)] $\mathrm{PMon}(q)\rtimes\Aut(\F_q)$ if $q=n+1$, where $\mathrm{PMon}(q)$ denotes the quotient group of the monomial matrices by the scalar matrices, having size $hq^{n+1}(q-1)^{n}q!$.
\end{itemize}
Moreover, if $q=n+2$ and $q$ is prime, then $\Aut(\Gamma_{n,q}(\K))$ is isomorphic to $\PGammaL(n+2,q)_\K$.

\end{theorem}
\begin{proof}
(i) If $q=n+2$, then $\K$ is projectively equivalent to the frame $\K'$ of $\PG(n,q)$ with points $P_1,\ldots,P_{n+2}$, where $P_i$ has coordinates $v_i$, and $v_1=(1,0,\ldots,0),v_2=$ $(0,1,0,\ldots,0)$, $\ldots,v_{n+1}=(0,\ldots,0,1),v_{n+2}=(-1,-1,\ldots,-1)$. Let $B_k=(b_{ij})_k$, $1\leq k\leq n+1$, be the matrix with $b_{ii}=1$, $i\neq k$, $1\leq i\leq n+1$, $b_{ik}=-1$, $1\leq i\leq n+1$, and $b_{ij}=0$ for all other $i,j$. Let $G_{per}$ denote the subgroup of permutation matrices of $\mathrm{GL}(n+1,q)$, and consider the subgroup $G$ of $\mathrm{GL}(n+1,q)$, generated by the elements of $G_{per}$ and the matrices$B_k$, $1\leq k\leq n+1$.

 For every matrix $B=(b_{ij})$, $1\leq i,j\leq n+1$, in $G$, we can define a matrix $A=(a_{ij})$, $0\leq i,j\leq n+1$, as the $(n+2)\times (n+2)$ matrix with $a_{00}=1$, $a_{i0}=a_{0j}=0$ for $i,j\geq 1$ and $a_{ij}=b_{ij}$ for $1\leq i,j\leq n+1$. Let $\widetilde{G}$ be the group obtained by extending all matrices of $G$ in this way. It is clear that the elements of $G$ are exactly the permutations of the elements of $\{v_1,\ldots,v_{n+2}\}$ and hence that $\widetilde{G}$ is isomorphic to $\PGL(n+1,q)_\K$. This already shows that $\PGL(n+1,q)_\K$ acts transitively on the points of $\K$, hence, by Theorem \ref{ET}, $\Gamma_{n,q}(\K)$ is edge-transitive.

It also follows that the only element of $\widetilde{G}$ fixing $\K$ pointwise corresponds to the identity matrix, which implies that any element of $\Persp(H_\infty)$ contained in $\widetilde{G}$ is trivial. Hence, $\PGL(n+2,q)_\K$ is isomorphic to $\Persp(H_\infty)\rtimes \PGL(n+1,q)_\K$.

(ii) If $q=n+1$, then $\PGL(n+1,q)$ is isomorphic to $\mathrm{SL}(n+1,q)$. Hence, $\PGL(n+1,q)$ can be embedded in $\PGL(n+2,q)_{H_\infty}$ by taking all matrices $B=(b_{ij})$, $1\leq i,j\leq n+1$, of $\mathrm{SL}(n+1,q)$ and, as before, defining $A=(a_{ij})$, $0\leq i,j\leq n+1$, with $a_{00}=1$, $a_{i0}=a_{0j}=0$ for $i,j\geq 1$ and $a_{ij}=b_{ij}$ for $1\leq i,j\leq n+1$. An element of $\Persp(H_\infty)$ corresponds to a matrix of the form $D=(d_{ij})$, $0\leq i,j\leq n+1$, with $d_{0j}=\lambda_j$, $0\leq j\leq n+1$, $d_{ii}=\mu$, $1\leq i\leq n+1$, for some $\lambda_j,\mu\in \F_q$, and $d_{ij}=0$ otherwise. This implies that the group $\widetilde{G}$ of matrices $A$ defined in this way meets $\Persp(H_\infty)$ trivially. Hence, $\PGL(n+2,q)_\K$ is isomorphic to $\Persp(H_\infty)\rtimes \PGL(n+1,q)_\K$.

Since $q=n+1$, the curve $\K$ is projectively equivalent to the set $\K'$ of points $P_1,\ldots,P_{n+1}$ in $\PG(n,q)$, where $P_i$ has coordinates $v_i$, and $v_1=(1,0,\ldots,0)$, $v_2=(0,1,0,\ldots,0)$, $\ldots,v_{n+1}=(0,\ldots,0,1)$. Using this, it is clear that $\PGL(n+1,q)_\K$ is isomorphic to the quotient group of monomial matrices by scalar matrices and that $\PGL(n+1,q)_\K$ acts transitively on $\K$. Hence, $\Gamma_{n,q}(\K)$ is an edge-transitive graph.

In both cases, it is clear that $\K'$ is stabilised by the Frobenius automorphism, hence, using Result \ref{isom}, it also follows that $\PGammaL(n+2,q)_{\K}\cong \Persp(H_\infty)\rtimes (\PGL(n+1,q)_\K\rtimes \Aut(\F_q))$. The observation on the sizes follows from $|\mathrm{Sym}(q)|=q!$ and $|\mathrm{PMon}|=|\mathrm{Sym}(q)|.|(\F_q^*)^n|/(q-1)=q!(q-1)^{n-1}$.

%%(3,4)!!
Since through every point of $H_\infty$ there is a tangent line to $\K$, Corollary \ref{NVT} shows that $\Gamma_{n,q}(\K)$ is not vertex-transitive. Since $\K$ spans $H_\infty$ and $|\K|=q$, we get that $\Gamma_{n,q}(\K)$ is semisymmetric.

The last part of the statement follows from Theorem \ref{hoofdstelling}.
\end{proof}

\begin{remark} For $n=2,q=3$, and $\K$ a basis is $\PG(2,3)$, we have showed, by using the computer program GAP \cite{GAP}, that all automorphisms are induced by a collineation of $\PG(3,3)$ so we have that the automorphism group of $\Gamma_{2,3}(\K)$ is isomorphic to $\PGammaL(4,3)_\K$. For $n=3,q=4$, however, again using the computer, we find that  [$\Aut(\Gamma_{n,q}(\K)):\PGammaL(n+2,q)_\K]=8$. This implies that there exist automorphisms of the graph $\Gamma_{3,4}(\K)$ that are not collineations of $\PG(4,4)$. For $n=4,q=5$, this index is already 7776. This might indicate that the general description of the full automorphism group of $\Gamma_{n,q}(\K)$, with $n+1=q$ is a hard problem.
\end{remark}

\subsection{$\K$ is contained in a normal rational curve and $q\geq n+3$}\label{nrcmin}

We will use the following theorem by Segre.
\begin{result}{\rm \cite{Segre}} \label{segre} If $q\geq n+2$, and $S$ is
a set of $n+3$ points in $\PG(n,q)$, no $n+1$ of which lie in a
hyperplane, then there is a unique normal rational curve in
$\PG(n,q)$, containing the points of $S$.
\end{result}
\begin{corollary} \label{nrc2} If $\K$ is a set of $q$ points of a normal rational curve $\mathcal{N}$ in $\PG(n,q)$, where $q\geq n+3$, then $\mathcal{N}$ is the unique normal rational curve through the points of $\K$.
\end{corollary}
The following theorem is well-known, a proof can be found in
e.g. \cite[Theorem 2.37]{Hughes}.

\begin{result} \label{nrc} If $q\geq n+2$ and $\mathcal{N}$ is a normal
rational curve in $\PG(n,q)$, then the stabiliser of $\mathcal{N}$ in
$\PGammaL(n+1,q)$, is isomorphic to $\PGammaL(2,q)$ (in its faithful action on the projective line).
\end{result}

These results enable us to give a construction for the following infinite two-parameter family of semisymmetric graphs.

%We require the following easy lemma.
%%%%
%\begin{lemma} \label{aut} If there is an element of $\PGammaL(n,q)$ mapping $\K$ to a point set $\K'$ that is stabilised under the Frobenius automorphism, then
%$\PGammaL(n,q)_\K\cong \PGL(n,q)_\K \rtimes \Aut(\F_q)$.
%\end{lemma}
%\begin{proof} Since $\K'$ is contained in the orbit of $\K$ under $\PGammaL(n,q)$, $\PGammaL(n,q)_\K\cong \PGammaL(n,q)_{\K'}$. Since all automorphisms of $\F_q$ are generated by the Frobenius automorhism, every automorphism of $\F_q$ stabilises $\K'$. If we restrict the well-known isomorphism $\PGammaL(n,q)\cong \PGL(n,q) \rtimes \Aut(\F_q)$ to elements of $\PGammaL(n,q)_{\K'}$, the theorem follows.
%\end{proof}
%%%%%

\begin{theorem}\label{structure} If $\K$ is a set of $q$ points, contained in a normal rational curve of $\PG(n,q)$, $q=p^h$, $n\geq 3$, $q\geq n+3$, or $n=2$, $q$ odd, then $\Gamma_{n,q}(\K)$ is a semisymmetric graph.

Moreover, $\Aut(\Gamma_{n,q}(\K))$ is isomorphic to $\Persp(H_\infty)\rtimes \mathrm{A\Gamma L}(1,q)$ and has size $hq^{n+2}(q-1)^2$.
\end{theorem}

\begin{proof} Since $|\K|=q$, $\Gamma_{n,q}(\K)$ is a $q$-regular graph. The set $\K$ is an arc in $\PG(n,q)$, spanning the space $\PG(n,q)$. It is clear that if $n\geq 3$, or if $q$ is odd, every point of $\PG(n,q)$ lies on at least one tangent line to $\K$. Hence, by Result \ref{connected}, Corollary \ref{NVT} and Theorem \ref{hoofdstelling}, $\Gamma_{n,q}(\K)$ is a connected non-vertex-transitive graph for which $\Aut(\Gamma_{n,q}(\K))\cong \PGammaL(n+2,q)_\K$. By Corollary \ref{nrc2}, $\K$ extends by a point $P$ to a unique normal rational curve $\N$.
%Denote the standard normal rational curve with coordinates $\{(0,\ldots,0,1)\}$ $\cup$ $\{  (1,t,t^2,t^3,\ldots,t^{n}) \mid t\in \F_q\}$, by $\mathcal{N}_{st}$. By definition, there is an element $\psi_1$ of $\PGL(n+1,q)$ mapping $\mathcal{N}$ onto $\mathcal{N}_{st}$. The image of $\K$ under $\psi_1$ is contained in $\mathcal{N}_{st}$, and since by Result \ref{nrc} the stabiliser of $\mathcal{N}_{st}$ in $\PGL(n+1,q)$ acts two-transitively on the points of $\mathcal{N}_{st}$, there is an element $\psi_2$ of $\PGL(n+1,q)$ such that $\psi_2\psi_1(\K)=\K_{st}=\{(1,t,t^2,t^3,\ldots,t^{n}) \mid t\in \F_q\}$, this is clearly fixed by $\Aut(\F_q)$.
Since $P$ must be fixed by the stabiliser of $\K$ and $\PGammaL(2,q)_P\cong \AGammaL(1,q)$, we get $\PGammaL(n+2,q)_\K\cong \Persp(H_\infty)\rtimes \mathrm{A\Gamma L}(1,q)$, by Result \ref{isom}. The size of this group follows when considering that  $|\Persp(H_\infty)|=q^{n+1}(q-1)$ and $|\mathrm{A\Gamma L}(1,q)|=hq(q-1)$. By Theorem \ref{ET} the graph $\Gamma_{n,q}(\K)$ is edge-transitive and thus semisymmetric.
\end{proof}

\subsection{$\K$ is contained in a non-classical arc in $\PG(3,q)$, $q$ even}\label{arcqeven}

The $(q+1)$-arcs in $\PG(3,q)$, $q$ even, have been classified, each of them has the same stabiliser group as the normal rational curve.

\begin{result}{\rm \cite{CasseGlynn}}\label{arcPG(3,q)even}
In $\PG(3, q)$, $q = 2^h$, $h > 2$, every $(q + 1)$-arc is projectively
equivalent to some $\C(\sigma)=\{(1,x,x^{\sigma},x^{\sigma+1}) \mid x \in \F_q \} \cup \{(0, 0, 0, 1)\}$
where $\sigma$ is a generator of $\Aut(\F_q)$.
\end{result}

\begin{result}{\rm\cite{Luneburg}}\label{grouparcPG(3,q)even}
In $\PG(3, q)$, $q = 2^h$, $h>2$, the stabiliser of $\C(\sigma)$ in $\PGammaL(4,q)$ is isomorphic to $\PGammaL(2,q)$ (in its faithful action on the projective line).
%The set $\C(\sigma)$ is a $(q + 1)$-arc if and only if $\sigma$ is a generator of $Aut(\F _q )$.
%If $\alpha$ and $\beta$ are distinct generators of $Aut(\F _q )$, then $\C(\alpha)$ and $\C(\beta)$ are (projectively) equivalent if and only if $\alpha=\beta^{-1}$.
\end{result}

Note that the case $q=4$ is already discussed in Section \ref{q=n}.

\begin{result}\label{casse}{\rm \cite{Casse}}
For any $k$-arc of $\PG(3,q)$, $q=2^h$, $h>1$, we have $k \leq q+1$.
\end{result}

\begin{result}\label{bruen}{\rm\cite{Bruen}}
Let $\K$ be any $k$-arc in $\PG(3, q)$, $q=2^h$.
If $ k>(q+ 4)/2$, then $\K$ is contained in a unique complete arc.
\end{result}

\begin{corollary}\label{uniquearcqeven}
Consider a $(q+1)$-arc $\C(\sigma)$ of $\PG(3,q)$, $q=2^h$, $h>2$. If $\K$ is a set of $q$ points contained in $\C(\sigma)$, then there is a unique $(q+1)$-arc through the points of $\K$, namely $\C(\sigma)$.
\end{corollary}

\begin{proof}
Using Result \ref{bruen}, since $q > (q+4)/2$, $q>4$, we find a unique complete arc through $\K$. This arc has size at most $q+1$ by Result \ref{casse} and thus is equal to $\C(\sigma)$.
\end{proof}

\begin{theorem}If $\K$ is a set of $q$ points contained in any $(q+1)$-arc of $\PG(3,q)$, $q\geq 8$ even, then $\Gamma_{3,q}(\K)$ is a semisymmetric graph.

Moreover, $\Aut(\Gamma_{3,q}(\K))$ is isomorphic to $\Persp(H_\infty)\rtimes \mathrm{A\Gamma L}(1,q)$ and has size $hq^{5}(q-1)^2$.
\end{theorem}

\begin{proof}
The proof goes in exactly the same way as the proof of Theorem \ref{structure}, by making use of Corollary \ref{uniquearcqeven} and Results \ref{arcPG(3,q)even} and \ref{grouparcPG(3,q)even}.
\end{proof}

%%%%%%%%

\subsection{$\K$ is contained in the Glynn arc in $\PG(4,9)$}\label{10arc}
In \cite{Glynn} David Glynn constructs an example of an arc of size $q+1$ in $\PG(4,9)$, which is not a normal rational curve. We call this $10$-arc the {\em Glynn arc (of size 10)}. He also shows that an arc in $\PG(4,9)$ of size $10$ is a normal rational curve or a Glynn arc.

\begin{result} \cite{Glynn} \label{groupGlynn} The stabiliser in $\PGammaL(5,9)$ of the Glynn arc of size $10$ in $\PG(4,9)$ is isomorphic to $\PGL(2,9)$.
\end{result}

\begin{result}\cite{Blokhuis}
A $k$-arc in $\PG(n, q)$, $n \geq 3$, $q$ odd and $k \geq \frac{2}{3}(q - 1) + n$ is contained in a unique complete
arc of $\PG(n,q)$.
\end{result}

\begin{corollary}\label{glynnext}
If $\K$ is a set of $9$ points contained in a Glynn $10$-arc $\C$ of $\PG(4,9)$, then $\K$ is contained in a unique $10$-arc, namely $\C$.
\end{corollary}

\begin{theorem}If $\K$ is a $9$-arc contained in a Glynn $10$-arc of $\PG(4,9)$, then $\Gamma_{4,9}(\K)$ is a semisymmetric graph.

Moreover, $\Aut(\Gamma_{4,9}(\K))$ is isomorphic to $\Persp(H_\infty)\rtimes \AGL(1,9)$ and has size $9^6 8^2$.
\end{theorem}
\begin{proof}
Since $|\K|=9$, $\Gamma_{4,9}(\K)$ is a $9$-regular graph. The set $\K$ is an arc in $\PG(4,9)$, spanning the space $\PG(4,9)$. It is clear that every point of $\PG(4,9)$ lies on at least one tangent line to $\K$. Hence, by Result \ref{connected}, Corollary \ref{NVT} and Theorem \ref{hoofdstelling}, $\Gamma_{4,9}(\K)$ is a connected non-vertex-transitive graph for which $\Aut(\Gamma_{4,9}(\K)\cong \PGammaL(6,9)_\K$. By Corollary \ref{glynnext}, $\K$ extends by a point $P$ to a unique Glynn 10-arc $\C$.
By Result \ref{isom} we have $\PGammaL(6,9)_\K \cong \Persp(H_\infty) \rtimes \PGammaL(5,9)_\K$.
Since $\PGL(2,9)_P\cong \AGL(1,9)$, we find $\PGammaL(6,9)_\K \cong \Persp(H_\infty) \rtimes \AGL(1,9)$. As before, the size easily follows.
By Theorem \ref{ET} the graph $\Gamma_{4,9}(\K)$ is edge-transitive and thus semisymmetric.
\end{proof}

\subsection{Using the dual arc construction}
Let $\K = \{P_1, \ldots, P_k\}$ be a $k$-arc in $\PG(n, q)$, $k \geq n + 4$.
Consider the respective coordinates $(a_{0j},\ldots,a_{nj})$ of $P_j$, $1\leq j\leq k$, then $(n + 1) \times k$-matrix $A=(a_{ij})$ determines a vector space (an MDS code) $V_1 = V(n+ 1,q)$, which is a subspace of $V(k,q)$.
The space $V_1$ has a unique orthogonal complement $V_2 = V(k - n - 1, q)$ in $V(k, q)$. Then $V_2$ is also an MDS code \cite[p. 319]{MacWilliams}. A $k$-arc $\hat{\K}=\{Q_l, \ldots, Q_k\}$ of $\PG(k - n - 2, q)$ with respective coordinates $(b_{0j},\ldots,b_{k-n-2, j})$ of $Q_j$, $1\leq j \leq k$, such that the
$(k-n- 1)\times k$-matrix $B=(b_{ij})$ generates $V_2$, is called a {\em dual $k$-arc} $\hat{\K}$ of the $k$-arc $\K$ \cite{dualarc}.

It should be noted that duality for arcs is a $1-1$-correspondence between {\em equivalence classes} of arcs, rather than a correspondence between arcs: with another ordering of $\K$ and choosing other coordinates for the points of $\K$, we obtain the same set of dual $k$-arcs.
%
%The construction of dual arcs is mostly presented via the correspondence between arcs and MDS-codes. But it is also possible to describe this construction directly, as seen below.
%
%\begin{definition}\cite{GGG}
%Let $\K = \{P_0, \ldots, P_{k-1}\}$ be a $k$-arc in $\PG(n, q)$, $k \geq n + 4$. Choose the coordinate system such that $P_0 = P(e_0), \ldots,P_{n}=P(e_n)$, and $P_{n+1+j}=P(a_{0j}, \dots, a_{nj})$ for $0 \leq j \leq k-n-2$. Then the set $\hat{\K} = \{Q_0, \ldots, Q_{k-1} \}$ of points of $\PG(k - n - 2, q)$ with corresponding coordinates $Q_0=P(e'_0),\ldots ,Q_{k-n-2}=P(e'_{k-n-2})$ and $Q_{k-n-1+j}=P(a_{j0}, \ldots, a_{j, k-n-2})$ for $0 \leq j \leq n$, is a $k$-arc in $\PG(k-n- 2, q)$. This arc is called a {\em dual $k$-arc of $\K$}.
%\end{definition}

\begin{result}{\rm\cite[Theorem 2.1]{LeoJef2}}\label{dualcoll}
A $k$-arc $\K$ in $\PG(n, q)$, $k \geq n + 4$, and a dual $k$-arc $\hat{\K}$ of $\K$
in $\PG(k - n - 2, q)$ have isomorphic collineation groups and isomorphic projective groups.
\end{result}

%\begin{result}{\rm\cite{LeoJef1}}
%A $k$-arc $\K$ in $\PG(n, q)$, $k \geq n + 4$, is complete if and only if for any
%dual $k$-arc $\hat{\K}$ in $\PG(k -n - 2, q)$ of $\K$ there is no $k'$-arc $\hat{\K}'$ in $\PG(k' - n - 2,
%q) > \PG(k - n - 2, q)$, $k' > k$, such that $\hat{\K}$ is the projection of $\hat{\K}' \backslash \PG(k' - k - 1,
%q)$ from $\PG(k'-k - 1, q)$ onto $\PG(k -n -2, q)$, where $\PG(k'- k - 1, q)$ is
%skew to $\PG(k - n - 2, q)$ and generated by $k' - k$ points of $\hat{\K}'$.
%\end{result}

The duality transformation maps normal rational curves to normal rational curves and
non-classical arcs to non-classical arcs.
This implies that the arcs in Sections \ref{arcqeven} and \ref{10arc} give rise to a different family of semisymmetric graphs. This follows from the following theorem.

\begin{theorem} \label{dual}Let $\K$ be a $q$-arc in $H_\infty=\PG(n,q)$, $q\geq n+4$, and let $\hat{\K}$ be a dual arc of $\K$ in $\hat{H}_\infty=\PG(q-n-2,q)$. Suppose that one of the groups $\PGammaL(n+1,q)_\K$ or $\PGammaL(q-n-1,q)_{\hat{\K}}$ fixes a point outside $\K$, $\hat{\K}$ respectively, and acts transitively on the points of $\K$, $\hat{\K}$ respectively, then $\Gamma_{n,q}(\K)$ and $\Gamma_{q-n-2,q}(\hat{\K})$ are semisymmetric, $\Aut(\Gamma_{n,q}({\K}))\cong \Persp(H_\infty)\rtimes \PGammaL(n+1,q)_\K$ and $\Aut(\Gamma_{q-n-2,q}(\hat{\K}))\cong \Persp(\hat{H}_\infty)\rtimes \PGammaL(n+1,q)_\K$.
\end{theorem}

\begin{proof} In the same way as before, using
%Since $|\K|=|\hat{\K}|=q$, $\Gamma_{n,q}(\K)$ and $\Gamma_{q-n-2,q}(\hat{\K})$ are $q$-regular graphs. The set $\hat{\K}$ is an arc in $\PG(q-n-2,q)$, spanning the space $\hat{H}_\infty$, so it is clear that every point of
% $\hat{H}_\infty$ lies on at least one tangent line to $\hat{\K}$. Hence, by
Result \ref{connected}, Corollary \ref{NVT} and Theorem \ref{hoofdstelling}, we see that $\Gamma_{n,q}(\K)$ and $\Gamma_{q-n-2,q}(\hat{\K})$ are connected non-vertex-transitive graphs for which $\Aut(\Gamma_{n,q}(\K))\cong\PGammaL(n+2,q)_\K$ and $\Aut(\Gamma_{q-n-2,q}(\hat{\K}))\cong\PGammaL(q-n,q)_{\hat{\K}}$.

Suppose w.l.o.g. that $\PGammaL(n+1,q)_\K$ fixes a point, then by Result \ref{isom}, $\PGammaL(n+2,q)_\K \cong \Persp(H_\infty)\rtimes \PGammaL(n+1,q)_\K$. The embedding of $\PGammaL(n+1,q)_\K$ in $\PGammaL(n+2,q)_\K$ used to show this result was constructed by adding a $1$ at the lower right corner of every matrix $B$ corresponding to an element $(B,\theta)$ of $\PGammaL(n+1,q)_\K$, for some $\theta \in \Aut(\F_q)$ to obtain a matrix $B'$ corresponding to an element $(B',\theta)$ of $\PGammaL(n+2,q)_\K$. This subgroup meets $\Persp(H_\infty)$ trivially, which implies that in the group of matrices defining elements of $\PGammaL(n+1,q)_\K$, no proper scalar multiple of the identity matrix occurs.
Now, from the isomorphism constructed in the proof of Result \ref{dualcoll}, it follows that the group $\PGL(q-n-1,q)_{\hat{\K}}$, which is isomorphic to $\PGL(n+1,q)_\K$, also contains no proper scalar multiple of the identity matrix. Hence, by embedding $\PGammaL(q-n-1,q)_{\hat{\K}}$ in $\PGammaL(q-n,q)$ in the same way (by adding a 1 at the lower right corner), we see that it meets $\Persp(\hat{H}_\infty)$ trivially.
This implies that $\PGammaL(q-n,q)_{\hat{\K}}\cong \Persp(\hat{H}_\infty)\rtimes \PGammaL(n+1,q)_\K$.
%Since $\Aut(\F_q)$ intersects $\Persp(\hat{H}_\infty)$ trivially, it also follows that $\PGammaL(q-n-2,q)_{\hat{\K}}\cong \Persp(\hat{H}_\infty)\rtimes \PGammaL(n+1,q)_\K$.

We know that $\PGammaL(n+1,q)_\K$ and $\PGammaL(q-n-1,q)_{\hat{\K}}$ are permutation isomorphic, hence, if one of them acts transitively on the points of $\K$ or $\hat{\K}$, so does the other. By Theorem \ref{ET}, the graphs $\Gamma_{n,q}(\K)$ and $\Gamma_{q-n-2,q}(\hat{\K})$ are edge-transitive and hence  semisymmetric.
\end{proof}

If we restrict ourselves in the previous theorem to elements of the projective groups, using Result \ref{isom2} we get the following corollary.

\begin{corollary}Let $\K$ be a $q$-arc in $H_\infty=\PG(n,q)$, $q\geq n+4$, and let $\hat{\K}$ be a dual arc of $\K$ in $\hat{H}_\infty=\PG(q-n-2,q)$. Suppose that one of the groups $\PGL(n+1,q)_\K$ or $\PGL(q-n-1,q)_{\hat{\K}}$ fixes a point outside $\K$, $\hat{\K}$ respectively, and acts transitively on the points of $\K$, $\hat{\K}$ respectively. Suppose $\PGammaL(n+1,q)_\K \cong \PGL(n+1,q)_\K \rtimes \Aut(\F_{q_0})$ or $\PGammaL(q-n-1,q)_{\hat{\K}}\cong \PGL(q-n-1,q)_{\hat{\K}} \rtimes \Aut(\F_{q_0})$ respectively, for $q_0=p^{h_0}$, $h_0 | h$ or $\PGammaL(n+1,q)_\K \cong \PGL(n+1,q)_\K $, $\PGammaL(q-n-1,q)_{\hat{\K}}\cong \PGL(q-n-1,q)_{\hat{\K}}$ respectively.  Then $\Gamma_{n,q}(\K)$ and $\Gamma_{q-n-2,q}(\hat{\K})$ are semisymmetric, $\Aut(\Gamma_{n,q}({\K}))\cong \Persp(H_\infty)\rtimes \PGammaL(n+1,q)_\K$ and $\Aut(\Gamma_{q-n-2,q}(\hat{\K}))\cong \Persp(\hat{H}_\infty)\rtimes \PGammaL(n+1,q)_\K$.
\end{corollary}

Consider the Glynn $10$-arc contained in $\PG(4,9)$ and take any point $P$ of this $10$-arc; if we project the arc from $P$ onto a $\PG(3, 9)$ skew to $P$, then we obtain a complete $9$-arc of $\PG(3, 9)$.
%%It can be checked that this $9$-arc is contained in a unique elliptic quadric of $PG(3, 9)$.
%
In \cite{Glynn} the author also shows that all complete $9$-arcs in $\PG(3, 9)$ can be obtained in this way, i.e. all complete $9$-arc of $\PG(3, 9)$ are equivalent. It follows from \cite{LeoJef1} that the complete $9$-arc in $\PG(3,9)$ is the dual of a $9$-arc that is contained in the Glynn arc in $\PG(4,9)$. If we apply Theorem \ref{dual} to the Glynn $10$-arc, we obtain the following corollary. The size of the automorphism group follows as before.

\begin{corollary}If $\K$ is a complete 9-arc of $\PG(3,9)$, then $\Gamma_{3,9}(\K)$ is a semisymmetric graph.
Moreover, $\Aut(\Gamma_{3,9}(\K))$ is isomorphic to $\Persp(H_\infty)\rtimes \AGL(1,9)$ and has size $9^5 8^2$.
\end{corollary}

We can also apply Theorem \ref{dual} to the arcs of Section \ref{arcqeven}.

\begin{corollary}Let $\K$ be an arc of size $q$ contained in any $(q+1)$-arc of $\PG(q-4,q)$, $q=p^h> 8$ even, then $\Gamma_{q-4,q}(\K)$ is a semisymmetric graph.

Moreover, $\Aut(\Gamma_{q-4,q}(\K))$ is isomorphic to $\Persp(H_\infty)\rtimes \mathrm{A\Gamma L}(1,q)$ and has size $hq^{q-2}(q-1)^2$.
\end{corollary}

\section{Families of semisymmetric graphs arising from other sets}

By Result \ref{main}, if $\K$ is a set of points such that its closure $\overline{\K}$ is the whole space $H_\infty$, then every automorphism of the graph $\Gamma_{n,q}(\K)$ is induced by a collineation of its ambient space $\PG(n+1,q)$. However, we don't need this property for the construction of semisymmetric graphs. From the results and theorems of Section \ref{section2}, the following theorem clearly follows.

\begin{theorem}
Let $\K$ be a point set of $H_\infty=\PG(n,q)$ of size $q$ spanning $H_\infty$ such that every point of $H_\infty \backslash \K$ lies on at least one tangent line to $\K$, and such that $\PGammaL(n+1,q)_\K$ acts transitively on the points of $\K$. Then the graph $\Gamma_{n,q}(\K)$ is a connected semisymmetric graph.
\end{theorem}

The subgroup of the automorphism group of the graph $\Gamma_{n,q}(\K)$ for which the elements are induced by collineations of the space $\PG(n+1,q)$ will be called the {\em geometric automorphism group} of $\Gamma_{n,q}(\K)$.

We now give some examples of semisymmetric graphs for which $\overline{\K}$ is contained in a subgeometry of $H_\infty$. In the first three examples $\overline{\K}$ is a Baer subgeometry, obviously this only works if we look at a projective space over a field of square order. We will also construct their geometric automorphism group.

\subsection{$\K$ is contained in an elliptic quadric}

Let $\pi$ be a Baer subspace $\PG(3,\sqrt{q})$, embedded in $H_\infty=\PG(3,q)$, $q$ a square. Let $\K$ denote the set of points of an elliptic quadric $Q^-(3,\sqrt{q})$ in $\pi$ with one point removed. This set $\K$ has $q$ points and clearly every point not in $\K$ lies on at least one tangent line to $\K$. One can choose coordinates such that the set
\[\{(1,x,x^{\sqrt{q}},x^{\sqrt{q}+1}) \mid x \in \F_q\} \cup \{(0,0,0,1)\}\]
is the elliptic quadric. If $\K$ is this set with the point $(0,0,0,1)$ removed, then clearly $\K$ is stabilised by $\Aut(\F_q)$.

We introduce the definition of a {\em cap} and some results.

\begin{definition}A {\em $k$-cap} in $\PG(n,q)$ is a set of $k$ points such that no $3$ points lie on a line.\end{definition}

\begin{result}\label{uniqO1} \cite{Barlotti} \label{barlotti} A $q$-cap in $\PG(3,\sqrt{q})$, $q$ an odd square, is uniquely extendable to an elliptic quadric $Q^{-}(3,\sqrt{q})$.
\end{result}

\begin{result}\label{uniqO2} \cite[Chapter IV]{Segre2}
In $\PG(3,\sqrt{q})$, $q>4$ an even square, a $k$-cap with $q-\sqrt[4]{q}/2 +1 < k < q+1$ lies on a unique complete $(q+1)$-cap.
\end{result}

\begin{result} \cite[Section 15.3]{fini3} The stabiliser in $\PGL(4,\sqrt{q})$ of an elliptic quadric in $\PG(3,\sqrt{q})$ is $\PGO^{-}(4,\sqrt{q})$, which is isomorphic to $\PGL(2,q)$ (in its faithful action on the projective line).
\end{result}

\begin{theorem} The graph $\Gamma_{3,q}(\K)$, $q>4$ square, is semisymmetric. Moreover, the geometric automorphism group is isomorphic to $\Persp(H_\infty)\rtimes \AGammaL(1,q)$ and has size $hq^{5}(q-1)^2$.
\end{theorem}
\begin{proof}
Since $\K$ consists of $q$ points spanning $\PG(3,q)$, $\Gamma_{3,q}(\K)$ is $q$-regular and is connected by Result \ref{connected}. The graph $\Gamma_{3,q}(\K)$ is not vertex-transitive by Corollary \ref{NVT}. The geometric automorphism group of $\Gamma_{3,q}(\K)$ is $\PGammaL(5,q)_\K$. By Results \ref{uniqO1} ($q$ odd) and \ref{uniqO2} ($q$ even), the cap $\K$ extends uniquely to an elliptic quadric in $\PG(3,\sqrt{q})$, and hence by Result \ref{isom}, we find $\PGL(5,q)_\K \cong \Persp(H_\infty) \rtimes \PGL(4,q)_\K$. The group stabilising the elliptic quadric also stabilises its ambient subgeometry $\PG(3,\sqrt{q})$ and $\K$ is fixed by $\Aut(\F_q)$, hence we find $\PGammaL(4,q)_\K \cong\PGL(4,\sqrt{q})_\K \rtimes \Aut(\F_q) \cong \AGammaL(1,q)$, by Result \ref{isom2}. Since $\AGammaL(1,q)$ acts transitively on the points of $\K$, the graph is semisymmetric.
The size of this group follows from $|\Persp(H_\infty)|=q^{4}(q-1)$ and $|\mathrm{A\Gamma L}(1,q)|=hq(q-1)$.
\end{proof}

\subsection{$\K$ is contained in a Tits-ovoid}

Let $\pi$ be a Baer subspace $\PG(3,\sqrt{q})$, embedded in $H_\infty=\PG(3,q)$, $q=2^{2(2e+1)}$. Let $\K$ denote the set of points of a Tits ovoid in $\pi$ with one point removed. This set $\K$ has $q$ points and forms a cap in $\PG(3,q)$.

 The canonical form of a Tits ovoid in $\PG(3,\sqrt{q})$, $\sqrt{q}=2^{{2e+1}}$ is
\[\{(1,s,t,st+s^{\sigma+2}+t^{\sigma}) \mid s,t \in \F_{\sqrt{q}} \} \cup \{(0,0,0,1)\},\]
where $\sigma: \F_{\sqrt{q}} \rightarrow \F_{\sqrt{q}} : x \mapsto x^{2^{e+1}}$.
Let the set $\K$ correspond to the points of this ovoid minus the point $(0,0,0,1)$, then $\K$ is clearly stabilised by $\Aut(\F_{q})$.
\begin{result}\cite{Tits}
The stabiliser of $\K$ in $\PGL(4,\sqrt{q})$ is the 2-transitive Suzuki simple group $\Sz(\sqrt{q})$.
\end{result}

Following the notation of \cite[Chapter 11]{Huppert}, the point stabiliser of $\Sz(\sqrt{q})$ will be denoted by $\mathfrak{FH}$. Since $\Sz(\sqrt{q})$ is 2-transitive, the group $\mathfrak{FH}$ is transitive.

\begin{theorem} The graph $\Gamma_{3,q}(\K)$,  $q=2^{2(2e+1)}$, is semisymmetric. Moreover, the geometric automorphism group is isomorphic to
$\Persp(H_\infty)\rtimes \mathfrak{FH} \rtimes \Aut(\F_q)$ and has size $hq^5(q-1)(\sqrt{q}-1)$.
\end{theorem}
\begin{proof}
The proof works in exactly the same way as for the elliptic quadric. The size of the group follows when considering that  $|\Persp(H_\infty)|=q^{4}(q-1)$ and $|\mathfrak{FH}|=q(\sqrt{q}-1)$.
\end{proof}
\subsection{$\K$ is contained in a hyperbolic quadric $Q^{+}(3,q)$}
Let $\pi$ be a Baer subspace $\PG(3,\sqrt{q})$, embedded in $H_\infty=\PG(3,q)$, $q>4$ square. Let $\K$ denote the set of points of a hyperbolic quadric $Q^+(3,\sqrt{q})$ in $\pi$ with two lines of different reguli removed. This set $\K$ has $q$ points.

\begin{result} \cite[Section 15.3]{fini3} The stabiliser in $\PGammaL(4,\sqrt{q})$ of a hyperbolic quadric in $\PG(3,\sqrt{q})$ is $\PGammaO^{+}(4,\sqrt{q})$, which is isomorphic to $(2 \times \PGL(2,\sqrt{q}) \times \PGL(2,\sqrt{q}))\rtimes \Aut(\F_{\sqrt{q}})$ for $\sqrt{q}>2$.
\end{result}

\begin{corollary}
For $\sqrt{q}>2$, the stabiliser in $\PGammaL(4,\sqrt{q})$ of a hyperbolic quadric in $\PG(3,\sqrt{q})$ fixing two lines of different reguli is isomorphic to $(2\times \AGL(1,\sqrt{q}) \times \AGL(1,\sqrt{q}))\rtimes \Aut(\F_{\sqrt{q}})$.
\end{corollary}

\begin{theorem} The graph $\Gamma_{3,q}(\K)$, $q=p^h>4$ square, is semisymmetric. Moreover, the geometric automorphism group is isomorphic to $\Persp(H_\infty)\rtimes (2\times\AGL(2,\sqrt{q})\times \AGL(2,\sqrt{q}))\rtimes \Aut(\F_q)$ and has size $2hq^5(q-1)(\sqrt{q}-1)^2$.
\end{theorem}
\begin{proof}
Since $\K$ consists of $q$ points spanning $\PG(3,q)$, $\Gamma_{3,q}(\K)$ is $q$-regular and is connected by Result \ref{connected}. Clearly every point of $\PG(3,q)$ not in $\K$ lies on at least one tangent to $\K$, hence $\Gamma_{3,q}(\K)$ is not vertex-transitive by Corollary \ref{NVT}. The geometric automorphism group is $\PGammaL(5,q)_\K$. Clearly $\K$ extends uniquely to a hyperbolic quadric in $\PG(3,\sqrt{q})$ by adding the missing line of each regulus. Since the intersection point of these lines will be fixed by the stabiliser of $\K$, we find by Result \ref{isom} that $\PGL(5,q)_\K\cong \Persp(H_\infty) \rtimes \PGL(4,q)_\K$. Since the group stabilising the hyperbolic quadric also stabilises the subgeometry $\PG(3,\sqrt{q})$ in which it lies and the canonical form of $Q^{+}(3,\sqrt{q})$ is fixed by $\Aut(\F_q)$, we find $\PGammaL(4,q)_\K\cong\PGL(4,\sqrt{q})_\K \rtimes \Aut(\F_q) \cong (2 \times \AGL(1,\sqrt{q})\times\AGL(1,\sqrt{q}))\rtimes\Aut(\F_{q})$, by Result \ref{isom2}. Since $2\times \AGL(1,\sqrt{q})\times\AGL(1,\sqrt{q})$ acts transitively on the points of $\K$, the graph is semisymmetric.
\end{proof}

\subsection{$\K$ is contained in a cone}

Let $\Pi$ be a subspace $\PG(n,q_0)$ embedded in $H_\infty=\PG(n,q=q_0^h)$. Let $\pi$ be a hyperplane of $\Pi$. Consider $\O$ a set of $q_0^{h-1}$ points of $\pi$. Let $V$ be a point of $\Pi \backslash \pi$, and let $V\O$ denote the set of points of the cone in $\Pi$ with vertex $V$ and base $\O$. This set minus its vertex $V$ has $q$ points.

For a vertex $v$ in a graph $\Gamma$ and a positive integer $i$ we write
$\Gamma_{i}(v)$ for the set of vertices at distance $i$ from $v$.

\begin{lemma}\label{NVT2}
Let $\K$ be the cone $V\O$ of $\Pi$ minus its vertex $V$, such that every point of $\pi \backslash \O$ lies on at least one tangent line to $\O$, then $\forall P \in \P, \forall L \in \L: \Gamma_{n,q}(\K)_{4}(P) \not\cong \Gamma_{n,q}(\K)_{4}(L)$.
\end{lemma}
\begin{proof}
We will prove that, for every line $L \in \L$, the set of vertices $\Gamma_{n,q}(\K)_{4}(L)$
contains more than $q-1$ vertices that have all their neighbours in
$\Gamma_{n,q}(\K)_{3}(L)$, while for every point $P \in \P$, their are exactly $q-1$ vertices in the set $\Gamma_{n,q}(\K)_{4}(P)$
that have all their neighbours in $\Gamma_{n,q}(\K)_{3}(P)$.

To prove the first claim, consider a line $L \in \L$ with $L \cap H_\infty= P_1 \in \K$. Choose an affine point $Q$ on $L$ and a point $P_2 \in \K$ different from $P_1$.
Take a point $R$ on $Q P_2$, not equal to $Q$ or $P_2$, then clearly the line $R P_1 \in \Gamma_{n,q}(\K)_{4}(L)$. We will show that $R P_1$ has all its neighbours in $\Gamma_{n,q}(\K)_{3}(L)$. Consider a neighbour $S$ of $R P_1$, i.e $ S \in R P_1 \setminus \{P_1\}$.
The line $S P_2$ meets $L$ in a point $T$. Since $T \in \Gamma_{n,q}(\K)_{1}(L)$ and $T P_2 \in \Gamma_{n,q}(\K)_{2}(L)$, it follows that $S \in \Gamma_{n,q}(\K)_{3}(L)$.
Clearly any line $M \in \L$ through $P_1$, such that $\langle M, L \rangle \cap H_\infty$ contains at least two points in $\K$, belongs to $\Gamma_{n,q}(\K)_{4}(L)$ and has all its neighbours in $\Gamma_{n,q}(\K)_{3}(L)$. Since the points of $\K$ do not lie on one line, there are more than $q-1$ such lines $M$.

 Consider now a point $P \in \P$ and a point $T \in \Gamma_{n,q}(\K)_{4}(P)$. Look at the following minimal path of length 4 from $T$ to $P$: the point $T$, a line $Q_1 P_1 \in \Gamma_{n,q}(\K)_{3}(P)$ containing $T$ for some $P_1 \in \K$, an affine point $Q_1 \in \Gamma_{n,q}(\K)_{2}(P)$, the line $P P_2 \in \Gamma_{n,q}(\K)_{1}(P)$ containing $Q_1$, for some $P_2 \in \K$ different from $P_1$, and finally the point $P$. Consider the point $R= PT \cap H_\infty$, then it follows from our construction that $R$ lies on the line $P_1P_2$. Since $PR \notin \Gamma_{n,q}(\K)_{1}(P)$, we have $R$ not in $\K$. First, suppose there is a tangent line of $\K$ through $R$, say $RP_3$, with $P_3\in \K$. The line $TP_3$ is a neighbour of $T$.
If $TP_3$ belongs to $\Gamma_{n,q}(\K)_{3}(P)$, then there exists a line $PT'$ through a point $P_4\in \K$, with $T'$ on $TP_3$, which implies that $RP_3$ contains the point $P_4\in \K$, a contradiction. Hence in this case there are neighbours of $T$ that do not belong to $\Gamma_{n,q}(\K)_3(P)$.
Now suppose there is no tangent line of $\K$ through $R$, then by construction, $R$ is the vertex $V$ of the cone. A line through $V$ either contains 0 or $q_0$ points of $\K$, so in this case, any neighbour of $T$ belongs to $\Gamma_{n,q}(\K)_{3}(P)$. There are exactly $q-1$ points on the line $VP$ different from $P$ and $V$.
\end{proof}

\begin{corollary}\label{cor22}
The graph $\Gamma_{n,q}(\K)$ is not vertex-transitive.
\end{corollary}
\begin{proof}
Since any graph automorphism preserves distance and hence neighbourhoods, no automorphism of $\Gamma_{n,q}(\K)$ can map a vertex in $\P$ to a vertex in $\L$.
\end{proof}

Denote the subgroup of $\PGammaL(n+1,q)$ consisting of the perspectivities with centre $V$ by $\Persp(V)$.
%%%hier!
\begin{lemma}\label{conePGL} Consider $\K$, the point set of the cone $V\O$ in $\PG(n,q_0)$, minus its vertex $V$, where $\O$ spans $\pi$. If $\PGammaL(n,q_0)_\O$ and $\PGL(n,q_0)_\O$, respectively, fix a point of $\pi$, then $\PGammaL(n+1,q)_{\K} \cong \Persp(V)\rtimes \PGammaL(n,q_0)_\O \rtimes (\Aut(\F_q) / \Aut(\F_{q_0}))$ and $\PGL(n+1,q)_{\K} \cong \Persp(V)\rtimes \PGL(n,q_0)_\O$, respectively.\end{lemma}
\begin{proof}
First, note that, since $\O$ spans $\pi$, $\K$ spans $\Pi$, so $\PGammaL(n+1,q)_{\K}$ and  $\PGL(n+1,q)_{\K}$ stabilise the subgeometry $\Pi$. This implies that $\PGammaL(n+1,q)_\K\cong (\PGammaL(n+1,q)_\Pi)_\K$, and $\PGL(n+1,q)_\K\cong (\PGL(n+1,q)_\Pi)_\K$ respectively. Since $\PGammaL(n+1,q)_\Pi$ is clearly isomorphic to $\PGL(n+1,q_0)\rtimes (\Aut(\F_q) / \Aut(\F_{q_0}))$, we have that $\PGammaL(n+1,q)_\K\cong \PGammaL(n+1,q_0)_\K \rtimes (\Aut(\F_q) / \Aut(\F_{q_0}))$. Also, since $\PGL(n+1,q)_\Pi$ is isomorphic to $\PGL(n+1,q_0)$, we have that $\PGL(n+1,q)_\K\cong \PGL(n+1,q_0)_\K$.

Let $\phi$ be an element of $\PGammaL(n+1,q_0)_\K$, then $\phi$ preserves the lines through $V$. Define the action of $\phi$ on $\pi$ to be the mapping taking $L\cap \pi$ to $\phi(L)\cap \pi$.

The kernel of this action of $\PGammaL(n+1,q_0)_\K$ on $\pi$ is clearly isomorphic to
$\Persp(V)$, as it consists of all collineations fixing the lines through $V$. The image of the action is isomorphic to
$\PGammaL(n,q_0)_\O$, showing that $\PGammaL(n+1,q_0)_{\K}$ is an extension of
$\Persp(V)$ by $\PGammaL(n,q_0)_\O$.
To show that this extension splits, we embed $\PGammaL(n,q_0)_\O$ in $\PGammaL(n+1,q_0)_{\K}$ in such a way that it intersects trivially with $\Persp(V)$. By assumption, $\PGammaL(n,q_0)_\O$ fixes a point $P\in \pi$.  W.l.o.g. let $\pi$ be the hyperplane with equation $X_0=0$ and let $V$ be the point $(1,0\ldots,0)$. Suppose that $P$ has coordinates $(0,c_1,c_2,\ldots,c_{n-1})$, where the first non-zero coordinate equals one. This implies that for each $\beta\in \PGammaL(n,q_0)_\O$, there exists a unique $n\times n$ matrix $B=(b_{ij})$, $1\leq i,j\leq n$, and $\theta \in \Aut(\F_{q_0})$ corresponding to $\beta$, such that $(c_1,c_2,\ldots,c_{n-1})^{\theta}.B=(c_1,c_2,\ldots,c_{n-1})$. Moreover, the obtained matrices $B$ form a subgroup of $\Gamma\mathrm{L}(n,q_0)$. Let $A_\beta=(a_{ij})$, $0\leq i,j\leq n$, be the $(n+1)\times (n+1)$ matrix with $a_{00}=1$, $a_{i0}=a_{0j}=0$ for $i,j\geq 1$ and $a_{ij}=b_{ij}$ for $1\leq i,j\leq n$. It is clear that the semi-linear map $(A_\beta,\theta)$ defines an element of $\PGammaL(n+1,q_0)_{\K}$, corresponding to a collineation $\alpha$ acting in the same way as $\beta$ on $H_\infty$. If $\theta \neq \mathbbm{1}$, then $\alpha$ is not a perspectivity. If $\theta = \mathbbm{1}$, then $\alpha$ fixes every point on the line through $P$ and $V$, thus fixes at least two affine points and hence is not a perspectivity. This implies that the elements $\alpha$ form a subgroup of $\PGammaL(n+1,q)_{\K}$ isomorphic to $\PGammaL(n,q_0)_\O$ and intersecting $\Persp(V)$ trivially. This implies that $\PGammaL(n+1,q_0)_\K\cong \Persp(V)\rtimes \PGammaL(n,q_0)_\O $, and we have seen before that $\PGammaL(n+1,q)_\K\cong \PGammaL(n+1,q_0)_\K \rtimes (\Aut(\F_q) / \Aut(\F_{q_0}))$. Since $\Persp(V)$ intersects trivally with the standard embedding of $\Aut(\F_q) / \Aut(\F_{q_0})$, the claim follows.

The claim for $\PGL(n+1,q)_\K$ can be proved in the same way.
\end{proof}

The following corollary follows easily when we take into account that $\Persp(V)$ acts transitively on the points of each line through $V$.
\begin{corollary}
If $\PGammaL(n,q_0)_\O$ acts transitively on $\O$, then $\PGammaL(n+1,q)_\K$ acts transitively on $\K$.
\end{corollary}

\begin{theorem} Suppose that $\O$ spans $\pi$, that every point of $\pi \backslash \O$ lies on a tangent line to $\O$ and that $\PGammaL(n,q_0)_\O$ acts transitively on $\O$. Then the graph $\Gamma_{n,q}(\K)$ is semisymmetric. Moreover, the geometric automorphism group is isomorphic to $\Persp(H_\infty)\rtimes \Persp(V)\rtimes \PGammaL(n,q_0)_\O \rtimes (\Aut(\F_q) / \Aut( \F_{q_0}))$.
\end{theorem}

\begin{proof}
Since $\K$ consists of $q$ points spanning $\PG(n,q)$, $\Gamma_{n,q}(\K)$ is $q$-regular and is connected by Result \ref{connected}. The graph $\Gamma_{n,q}(\K)$ is not vertex-transitive by Lemma \ref{NVT2}. Clearly $\PGammaL(n+1,q)_\K$ stabilises the point $V$, so we find by Result \ref{isom} that $\PGL(n+2,q)_\K\cong \Persp(H_\infty) \rtimes \PGL(n+1,q)_\K$. The expression for the geometric automorphism group follows from Lemma \ref{conePGL}.
Since $\PGammaL(n+1,q)_\K$ acts transitively on the points of $\K$, by Theorem \ref{ET}, the graph is edge-transitive, and hence semisymmetric.
\end{proof}

\section{Isomorphisms of $\Gamma_{n,q}(\K)$ with other graphs}
\label{isomorphisms}
In this section, we will show that the graphs constructed by Du, Wang and Zhang \cite{Du}, and the graphs of Lazebnik and Viglione \cite{Felix} belong to the family $\Gamma_{n,q}(\K)$, where $\K$ is a $q$-arc contained in a normal rational curve (see Section \ref{nrcmin}).

\subsection{The graph of Du, Wang and Zhang}
If $q=p$ prime, then the point of $\PG(n,q)$ with coordinates $(0, \ldots, 0,1)$ and the orbit of the point $P$ with coordinates $(1,0,\ldots,0)$ under the element $\phi \in \PGL(n+1,p)$ of order $p$, defined by the matrix $A_\phi$,
 form a normal rational curve $\N$ in $\PG(n,p)$ (see
e.g. \cite{Leo}):

\[A_\phi = \left[\begin{matrix}1&1&0&0&\ldots&0&0\\
0&1&1&0&\ldots&0&0\\
0&0&1&1&\ldots&0&0\\
&&&\ldots&&&\\
0&0&0&0&\ldots&1&1\\
0&0&0&0&\ldots&0&1
\end{matrix}\right].\]

When we use the orbit of $P$ for the point set $\K$ at
infinity, we obtain a reformulation of the construction of the semisymmetric
graphs found by Du, Wang and Zhang in \cite{Du}. This shows that our
construction of the graph $\Gamma_{n,q}(\K)$, with $\K$ a set of $q$ points, contained in a normal rational curve, contains their family (and extends their construction to the case where $q$ is not a prime).
Moreover, the edge-transitive group of automorphisms described by the authors is not
the full automorphism group of the graph: they only consider
automorphisms induced by the group $\langle\phi\rangle$ of order $p$
acting on the points of $\K$, together with $\Persp(H_\infty)$.

\subsection{The graph of Lazebnik and Viglione}\label{subsection Felix}

In \cite{Felix}, the authors define the graph $\Lambda_{n,q}$ as
follows. Let $\P_{n}$ and $\L_{n}$ be two $(n+1)$-dimensional vector
spaces over $\F_q$. The vertex set of $\Lambda_{n,q}$ is
$\P_{n}\cup \L_{n}$, and we declare a point $(p) = (p_1, p_2,\ldots ,
p_{n+1})$ adjacent to a line $[l] = [l_1, l_2, \ldots, l_{n+1}]$ if and only if the
following $n$ relations on their coordinates hold.
\begin{eqnarray*}
l_2+p_2&=& p_1l_1\\
l_3+p_3&=& p_1l_2\\
&\vdots&\\
l_{n+1}+p_{n+1}&=&p_1l_{n}
\end{eqnarray*}

In the following theorem, we will show that the graph $\Lambda_{n,q}$
is isomorphic to the graph $\Gamma_{n,q}(\K)$, where $\K$ is contained in a normal rational curve; hence, $\Gamma_{n,q}(\K)$
provides an embedding of the Lazebnik-Viglione graph in $\PG(n+1,q)$.
Note that in \cite{Felix}, the authors provide some automorphisms,
acting on the graph $\Lambda_{n,q}$, to show that this graph is
semisymmetric. From the isomorphism with $\Gamma_{n,q}(\K)$ it follows that
$\PGammaL(n+2,q)_\K$ is also the full automorphism group of the Lazebnik-Viglione
graph when $q \geq n+3$ or $q=p=n+2$.

\begin{theorem} $ \Lambda_{n,q}\cong \Gamma_{n,q}(\K)$, where $\K$ is a $q$-arc contained in a normal rational curve.
\end{theorem}
\begin{proof} The graph $\Lambda_{n,q}$ is isomorphic to the graph
$\Lambda_{n,q}'$ where the role of points and lines is reversed.
Hence, let $\Lambda_{n,q}'$ be the graph, where $(p_1,\ldots,p_{n+1})$ is
incident with $(l_1,\ldots,l_{n+1})$ if and only if
$p_{i+1}+l_{i+1}=l_1p_i$ for all $1\leq i \leq n$. Let
$\ell=(l_1,\ldots,l_{n+1})$ be a vertex of $\Lambda_{n,q}'$, then the
points, incident with $\ell$ form a line of $\AG(n+1,q)$: suppose
$(p_1,\ldots,p_{n+1})$ and $(p_1',\ldots,p_{n+1}')$ are vertices, adjacent with
$\ell$, then so is the vertex
$(p_1+\lambda(p_1'-p_1),\ldots,p_{n+1}+\lambda(p_{n+1}'-p_{n+1}))$, for any
$\lambda\in\F_q$.

Now let $(p_1,\ldots,p_{n+1})$ and $(p_1',\ldots,p_{n+1}')$ be vertices of
$\Lambda_{n,q}'$ and embed these points of $\AG(n+1,q)$ in $\PG(n+1,q)$,
by identifying $(p_1,\ldots,p_{n+1})$ with $(1,p_1,\ldots,p_{n+1})$. The line
$L$ determined by these points meets the hyperplane at infinity with equation $X_0=0$ of
$\AG(n+1,q)$ in the point $P_\infty=(0,p_1-p_1',\ldots,p_{n+1}-p_{n+1}')$. Now
the affine point set of $L$ is a vertex of $\Lambda_{n,q}'$ if
and only if there is an element $(l_1,\ldots,l_{n+1})\in\L_{n}$ such that
for all $1\leq i\leq n$\begin{eqnarray*}
p_{i+1}+l_{i+1}&=&l_1p_i\\
p_{i+1}'+l_{i+1}&=&l_1p_i'.
\end{eqnarray*}
This implies that $p_{i+1}-p_{i+1}'=l_1(p_{i}-p_{i}')$ for some $l_1\in
\F_q$ and for all $1\leq i\leq n$. Hence, the point $P_\infty$ has
coordinates $(0,1,l_1,l_1^2,\ldots,l_1^{n})$, which implies that all
the vertices $(l_1,\ldots,l_{n+1})$ of $\Lambda_{n,q}'$ define a line in
$\PG(n+1,q)$ through a point of the standard normal rational curve $\K$, minus the point $(1,0,\ldots,0)$. This is exactly the description of
the graph $\Gamma_{n,q}(\K)$.
\end{proof}

\begin{corollary} The automorphism group $\Aut(\Lambda_{n,q})$ of the graph $\Lambda_{n,q}$ is isomorphic to the edge-transitive group $\PGammaL(n+2,q)_\K$. Moreover
\begin{itemize}
\item If $q\geq n+3$, $q=p^h$, $p$ prime, $n\geq 3$ or $n=2$ and $q$ odd, then $\Aut(\Lambda_{n,q})$ has size $hq^{n+2}(q-1)^2$;
\item If $q=p=n+2$, then $\Aut(\Lambda_{n,q})$ has size $q^{n+1}(q-1)q!$.
%\item If $q=p=n=3$, then  $\Aut(\Lambda_{n-1,q})$ has size $q^n(q-1)^{n}q!=1296$.
\end{itemize}
\end{corollary}

%\begin{remark} For $n=2,q=3$, $\PGammaL(4,q)_\K$ has size $q^n(q-1)^{n}q!=1296$. By computer, we showed that this is the size of the full automorphism group, hence, that $\Aut(\Gamma_{n,q}(\K)$ is isomorphic to $\PGammaL(n+1,q)_\K$. It is also possible to show this by hand, using the arguments of \cite{wij2}.
%\end{remark}

\subsection{The graph of Wenger and cycles in $\Gamma_{n,q}(\K)$}

We use the symbol $C^k$ for a cycle of length $k$. The infinite family of graphs $H_n(q)$ introduced in \cite{Felix2} and \cite{Wenger} are clearly isomorphic to the graphs $\Lambda_{n-1,q}$ of Section \ref{subsection Felix}, and thus isomorphic to the graphs $\Gamma_{n-1,q}(\K)$, where $\K$ is a $q$-arc contained in a normal rational curve.
Wenger \cite{Wenger} proved that the graphs $H_2(p)$, $H_3(p)$, $H_5(p)$ respectively, do not contain a $C^4$, $C^6$, $C^{10}$, respectively, for any prime $p$.
In \cite{Felix2} the authors notice that, for a prime power $q$ (implicitly assuming $n\geq 5$), the graph $H_n(q)$ contains no $C^{10}$ and prove it has girth 8 for $n\geq 3$.

We now prove a similar theorem for the graph $\Gamma_{n,q}(\K)$ using its geometric properties.
\begin{theorem}\label{no C4,C6,C10} Let $\K$ be any arc in $\PG(n,q)$, then the graph $\Gamma_{n,q}(\K)$ does not contain a $C^4$, $C^6$ and has girth 8. If $n=2$ and $|\K|\geq 4$, then $\Gamma_{n,q}(\K)$ contains cycles of length $10$. If $n\geq 3$, the graph $\Gamma_{n,q}(\K)$ is $C^{10}$-free.
\end{theorem}

\begin{proof} Since $\Gamma_{n,q}(\K)$ is bipartite, every cycle has even length. Note that a cycle $C^{2k}$ of $\Gamma_{n,q}(\K)$ contains $k$ points of $\P$ and $k$ lines of $\L$. Since there is at most one line of $\L$ through any two affine points, the graph does not contain a $C^4$. Suppose $\Gamma_{n,q}(\K)$ contains a $C^6$, $R_1 \sim R_1 R_2 \sim R_2  \sim R_2 R_3 \sim R_3\sim R_3 R_1$, $R_i \in \P$, $R_i R_j \in \L$. Clearly the affine points $R_1, R_2, R_3$ are not collinear. The plane $\langle R_1, R_2, R_3 \rangle$ intersects $H_\infty$ in a line. The lines $R_1 R_2$, $R_2 R_3$ and $R_3 R_1$ define three different points of $\K$, all lying on this line, a contradiction since $\K$ is an arc.

Consider two points $P_1, P_2 \in \K$ and a plane $\pi$ through $P_1 P_2$ not contained in $H_\infty$. For $i=1,2$ consider distinct lines $L_i$ through $P_1$ and distinct lines $M_i$ through $P_2$, different from $P_1 P_2$. Define the intersection points $R_{i j}=L_i \cap M_j$. The path $R_{11} \sim L_1 \sim R_{12} \sim M_2 \sim R_{22} \sim L_2 \sim R_{21} \sim M_1$ is a cycle $C^{8}$.
Since $\Gamma_{n,q}(\K)$ does not contain a $C^4$ or $C^6$, it has girth 8.

Let $\K$ be an arc in $\PG(2,q)$, and let $P_1,P_2,P_3,P_4$ be four points of $\K$. Let $R_1$ be an affine point. Let $\pi$ be a plane through $P_3P_4$, not through $R_1$. Let $R_2$ be $\pi\cap R_1P_2$ and $R_5$ be $\pi\cap R_1P_1$. Let $R_3$ be an affine point on $R_2P_3$, different from $R_2$ and let $R_4$ be the point $R_3P_4\cap R_5P_3$. Then $R_1 \sim R_1 R_2 \sim \cdots \sim R_5\sim R_5 R_1$, $R_i \in \P$, $R_i R_j \in \L$, is a cycle of length $10$.

Now let $n\geq 3$, let $\K$ be an arc and assume $\Gamma_{n,q}(\K)$ contains a $C^{10}$, $R_1 \sim R_1 R_2 \sim \cdots \sim R_5\sim R_5 R_1$, $R_i \in \P$, $R_i R_j \in \L$. %Clearly the points $R_1, R_2, R_3$ are not collinear. The plane $\pi = \langle R_1, R_2, R_3 \rangle$ intersects $H_\infty$ in a line containing exactly two points of $\C$, namely $\infty(R_1 R_2)$ and $\infty(R_2 R_3)$. Clearly, the points $R_4$ and $R_5$ are not both contained in this plane $\pi$. Suppose $R_4$ is contained in $\pi$, then $\langle \pi, R_5 \rangle$ is a 3-space intersecting $H_\infty$ in a plane containing four distinct points $\infty(R_1 R_2)$, $\infty(R_2 R_3)$, $\infty(R_4 R_5)$, $\infty(R_5 R_1)$ of $\C$, a contradiction since $\C$ is an arc. Similarly, the affine point $R_5$ is also not contained in $\pi$.
Note that two lines at distance 2 intersect $H_\infty$ in different points of $\K$; hence the five lines intersect $H_\infty$ in at least three different points of $\K$. The space $\pi = \langle R_1, R_2, R_3, R_4, R_5 \rangle$ has dimension at most 4 and at least 3, so intersects $H_\infty$ in at most a 3-space, containing at most 4 points of $\K$. Hence there are at least two lines of our set intersecting in a point of $\K$, these lines are not at distance two of each other, so without loss of generality, assume these are the lines $R_1 R_2$ and $R_3 R_4$. It follows that $\pi$ is a 3-space, intersecting $H_\infty$ in a plane containing 3 points of $\K$. However, the points $R_1, R_2, R_3$ and $R_4$ lie in a plane containing two points $P_1$ and $P_2$ of $\K$. The point $R_5$ does not lie in this plane, so the lines $R_4 R_5$ and $R_5 R_1$ intersect $H_\infty$ in two new points $P_3$ and $P_4$. The points $P_1, P_2, P_3$ and $P_4$ lie in a plane of $H_\infty$, a contradiction since $\K$ is an arc.
\end{proof}

{\bf Acknowledgment:} The authors want to thank Tim Penttila for his helpful comments regarding the proof of Theorem \ref{apart} (ii).

\end{document}